\documentclass[a4paper,11pt]{amsart}

\let\oldtocsection=\tocsection
\let\oldtocsubsection=\tocsubsection
\let\oldtocsubsubsection=\tocsubsubsection
 
\renewcommand{\tocsection}[2]{\hspace{0em}\oldtocsection{#1}{#2}\bfseries}
\renewcommand{\tocsubsection}[2]{\hspace{1.8em}\oldtocsubsection{#1}{#2}}
\renewcommand{\tocsubsubsection}[2]{\hspace{4.4em}\oldtocsubsubsection{#1}{#2}}
\makeatletter
\renewcommand\subsection{\@startsection{subsection}{2}%
  \z@{-.5\linespacing\@plus-.7\linespacing}{.5\linespacing}%
  {\normalfont\scshape}}
\renewcommand\subsubsection{\@startsection{subsubsection}{3}%
  \z@{.5\linespacing\@plus.7\linespacing}{.5\linespacing}%
  {\normalfont\scshape}}

\makeatother
\usepackage[framemethod=tikz]{mdframed}

\usepackage[pdftex]{hyperref}
\hypersetup{
    colorlinks=true, %set true if you want colored links
    linktoc=all,     %set to all if you want both sections and subsections linked
    linkcolor=blue,  %choose some color if you want links to stand out
    allcolors=blue,
}

\usepackage[utf8]{inputenc}
\usepackage[T1]{fontenc}
\usepackage{tikz-cd}
\usepackage{paralist}
\usepackage{mathrsfs}
\usepackage{amssymb}
\usepackage{amsmath}
\usepackage{amsfonts}
\usepackage{enumitem}

\newtheorem{theorem}{Theorem}[section]
\newtheorem{corollary}[theorem]{Corollary}
\newtheorem{lemma}[theorem]{Lemma}

\newtheorem{question}[theorem]{Question}

\newtheorem{remark}[theorem]{Remark}
\newtheorem{Fact}[theorem]{Fact}

\theoremstyle{definition}
\newtheorem{definition}[theorem]{Definition}

\usepackage[pdftex]{hyperref}
\hypersetup{
    colorlinks=true, %set true if you want colored links
    linktoc=all,     %set to all if you want both sections and subsections linked
    linkcolor=blue,  %choose some color if you want links to stand out
    allcolors=blue,
}

\newcommand{\brm}{\begin{remark}\begin{rm}}
\newcommand{\erm}{\end{rm}\end{remark}}

\newcommand{\bce}{\begin{compactenum}}
\newcommand{\ece}{\end{compactenum}}

\newcommand{\bb}[1]{\mathbb{#1}}
\newcommand{\cl}[1]{\mathcal{#1}}

\newcommand{\set}[2]{\{#1 \: ; \: #2\}}
\newcommand{\seq}[2]{\la#1 \: ; \: #2\ra}

\newcommand{\N}{\bb{N}}

\newcommand{\Ac}{\mathcal{A}}

\newcommand{\B}{\cl{B}}

\newcommand{\ZFC}{{\sf ZFC}}
\newcommand{\ZF}{{\sf ZF}}
\newcommand{\CH}{{\sf CH}}
\newcommand{\GCH}{{\sf GCH}}

\renewcommand{\P}{\bb{P}}

\newcommand{\sub}{\subseteq}

\newcommand{\la}{\langle}
\newcommand{\ra}{\rangle}

\newcommand{\Kr}{Kraj{\'i}{\v c}ek~}
\newcommand{\Krr}{Kraj{\'i}{\v c}ek}
\newcommand{\R}{\mathbb{R}}

\newcommand{\Z}{\mathbb{Z}}

\newcommand{\id}{\mathrm{id}}

\newcommand{\AcOm}{\mathcal{A}^+_{M,\Omega}}

\newcommand{\Fa}{\mathcal R_{\mathrm{max}}}
\newcommand{\MGa}{M[G]^{\Fa}}
\newcommand{\MGF}{M[G]^{\mathcal R}}
\newcommand{\Ap}{\mathcal{A}^+_{M,\Omega}}
\newcommand{\Apzero}{\mathcal{A}_{M,\Omega}}

\newcommand{\mup}{\mu_{M,\Omega}}

\newcommand{\Iup}{I_{M,\Omega}}

\newcommand{\st}{\mathit{st}}

\newcommand{\BB}{\mathcal{B}_{M,\Omega}}

\newcommand{\BS}{\mathcal{B}_{\omega_2}}
\newcommand{\BK}{\mathcal{B}_{\omega_1}}
\newcommand{\BM}{(\mathcal{B}, \bar{\mu})}
\newcommand{\BMD}{\mathcal{B}}
\newcommand{\Bk}{\mathcal{B}_\kappa}

\newcommand{\Sp}{\Sigma^+}

\newcommand{\rng}{\mathrm{range}}

\newcommand{\dotM}{\dot{M}^{\mathcal R}}
\newcommand{\dotMG}{M[G]^{\mathcal R}}

\begin{document}

\title[Forcing with random variables \ldots]{Forcing with random variables in bounded arithmetics and set theory}

\author{Radek Honzik}
\address{Charles University, Department of Logic,
Celetn{\' a} 20, Prague~1, 
116 42, Czech Republic
}
\email{radek.honzik@ff.cuni.cz}
\urladdr{logika.ff.cuni.cz/radek}

\thanks{
The author was supported by GA{\v C}R grant ``The role of set theory in modern mathematics'' (24-12141S)}

\begin{abstract} We analyse \Krr's  Boolean-valued random forcing $\BB$ in bounded arithmetics developed in \cite{krajicek11} from the perspective of the forcing in set theory.  We first observe that under the assumption that $M$ is a non-standard $\omega_1$-saturated model of true arithmetics of size $\omega_1$, and $\Omega \in M$ is a non-standard number, $\BB$ is isomorphic to the  probability (random) algebra corresponding to the product measure space on $2^{\omega_1}$ (and hence does not depend on $M$ and $\Omega$). This accentuates the analogies between \Krr's forcing and Scott's proof of the consistency of the negation of the Continuum Hypothesis in \cite{scott67} (Scott constructed a Boolean-valued model using an algebra corresponding to the product measure space $2^{\omega_2}$ to add $\omega_2$-many ``random reals'', each added by the separable algebra corresponding to $2^\omega$).  Thus, in a well-defined sense, the forcing $\BB$ adds a ``random integer'' to the model $M$, using a non-separable algebra corresponding to $2^{\omega_1}$. This puts \Krr's forcing---the only known example of a forcing construction in bounded arithmetics which adds ``generic'' integers---in the context of well-studied forcings in set theory. If $G$ is a generic filter for $\BB$ over a transitive model of set theory $V$, we naturally define in $V[G]$ two-valued generic extensions $\MGF$ of $M$ which correspond to Boolean-valued models in \cite{krajicek11} (where $\mathcal R$ ranges over collections of random variables which function as names for new integers). All $\MGF$'s are submodels of a maximal model $\MGa$ which uses all random variables (they satisfy the same quantifier-free formulas). We then study the relationship between the linear order $(M,<)$ and its extensions $(\MGF,<)$, proving several results on the extent of the mutual density of  new integers in $\MGF$ and the ``ground-model'' integers in $M$. We do not directly obtain new results in propositional complexity, but the presented results provide new insights into the structure of generic models of bounded arithmetics. At the end, we discuss some advantages and limitations of interpreting forcing in bounded arithmetics (and other weak theories) in the framework of set-theoretic forcing, providing an alternative to an axiomatic approach to forcing in bounded arithmetics formulated by Atserias and M{\"u}ller in \cite{atseriasmuller15}.\end{abstract}

\maketitle

\tableofcontents

\section{Introduction}

In \cite{krajicek11} \Kr developed a method which uses a probability measure algebra (derived from Loeb measures for hyperfinite integers in non-standard analysis) to construct Boolean-valued models of weak arithmetical theories. In an exciting example of a transfer of  ideas between different mathematical fields, he reformulated an argument of Scott \cite{scott67}---who used Boolean-valued models and a measure algebra to add new reals and thus show the consistency of the negation of the Continuum Hypothesis ($\CH$)---to find a way of adding new non-standard integers to saturated models of true arithmetics. By suitably limiting the collection of names for these new integers (``random variables'' in the terminology of \cite{krajicek11}) in terms of their computation complexity, he was able to prove independence of various statements in bounded arithmetics (with implications for the existence or non-existence of lower bounds for proofs of propositional tautologies, see \cite{buss85,krajicek95,cook2010logical} for more details on this connection). 

In this article, we analyze  \Krr's forcing construction from the point of forcing in set theory: we interpret it as providing a method for constructing various ``generic'' extensions\footnote{See Definitions (\ref{def:MG}) and (\ref{m:1}) for the exact meaning of $\MGF$.} $\MGF$ of  an $\omega_1$-saturated model $M$ of true arithmetics, where $\mathcal R$ is a collection of random variables.\footnote{The models we obtain are traditional two-valued models, not Boolean-valued models over uncountable Boolean algebras as in \cite{krajicek11}.} Since \Krr's forcing is the only genuine example of a forcing notion in bounded arithmetics which adds ``generic'' non-standard numbers to $M$, we will focus on proving some results related to the relationship of the linear order $(M,<)$ and its extension $(\MGF,<)$. 

Suppose for simplicity that $\CH$ holds and let $M$ be an $\omega_1$-saturated model of true arithmetics of size $\omega_1$. We start by ``connecting the dots'' in the existing literature and observe that  under this assumption the random forcing in \cite{krajicek11} is actually isomorphic to the probability measure algebra $\BK$ generated by the  product measure space $2^{\omega_1}$ with the standard measure (Theorem \ref{th:BBtype}).\footnote{\Kr does not specify the size of $M$, only the requirement of $\omega_1$-saturation is essential. We use $\CH$ for simplicity to fix a specific size of $M$. If $M$ is larger, the arguments in the article still apply but some additional details must be checked to find the corresponding probability algebra (see Corollary \ref{cor:BBtype}). However, these more general cases are not essential for this article so we will not pursue this topic systematically.} This observation is a combination of Maharam theorem \cite{maharam1942homogeneous}, results in non-standard analysis (hyperfinite integers and sets) due to Jin and Keisler \cite{jin1992maharam} and verification that the key properties still apply for \cite{krajicek11}. We find Theorem \ref{th:BBtype} intriguing, both from the strictly technical perspective of set theory, but also in the broader sense of utility of set theory in mathematics.

With regard to the technical perspective, since $\BK$ is a standard set-theoretic forcing,  it has been intensively studied in set theory. \Krr's construction---as discussed in this article---adds the interesting information that in generic extensions $V[G]$ for $\BK$, one also finds two-valued generic extensions $\MGF$ of bounded arithmetics which extend non-standard $\omega_1$-saturated models of true arithmetics  $M$ (see Corollary \ref{cor:BBtype}).

As regards the role of set theory, the set-theoretic view on forcing in bounded arithmetics may work as one particular way of providing a forcing framework for weak theories: instead of starting with a ground model $M$ of a weak theory $T$ and developing a forcing framework specific for $T$ (either with or without actual models), one can always start with the whole model $V$ of set theory, view $M$ as existing in $V$, and discuss extensions of $M$ in $V[G]$ (for an appropriate forcing notion $\P \in V$).

One can raise various objections against this perspective, the most serious one being that the explicit use of models of set theory kills important differences and nuances expressible in $T$ (like the schema of induction in bounded arithmetics). There is a priori truth in this position, but practice shows, so far, perhaps surprisingly, that the forcing notions used in bounded arithmetics are not specific to this weak theory: forcing with finite conditions over a countable non-standard model as in Paris--Wilkie \cite{pariswilkiewoods88} is set-theoretically isomorphic to adding a new Cohen real (see Section \ref{sec:PW}), and by Theorem \ref{th:BBtype}, \Krr's forcing is isomorphic to the random algebra $\BK$. It is plausible that other well-known forcing notions in set theory can prove to be useful in bounded arithmetics, provided specific combinatorial lemmas for $T$ and the respective forcing notions can be found, such as various switching lemmas in bounded arithmetics\footnote{\label{ft:sw}The switching lemmas are based on combinatorial lemmas in computational complexity. For example, the switching lemma for PHP-trees developed in~\cite{krajivcek2008exponential, pitassi1993exponential} to improve upon Ajtai's lower bound \cite{ajtai1994complexity} is analogous to an earlier result of Håstad~\cite{hastad1986almost} related to Boolean circuits. For more details, see an overview of Thapen~\cite{thapen2022notes}. See also Honzik et al.\ \cite{Rjoint} for more context and details.}. We give a few more comments on this speculative topic in Section \ref{sec:found}.

We attempted to make this article relatively self-contained and accessible to everyone with basic understanding of forcing in set theory. For this reason, we briefly review the definition of the algebra $\BB$ and the concept of random variables from \Kr \cite{krajicek11}. We also  summarize the key aspects of forcing with probability (measure) algebras and the essentials of non-standard measure theory.

The article is structured as follows. We first discuss non-standard models in forcing extensions in a general setting (Section \ref{sec:1}). A simple illustration of this approach is given in Section \ref{sec:PW} where we briefly mention the Paris--Wilkie forcing. In Section \ref{sec:PM} we discuss Maharam theorem and  how measure algebras are used in forcing (Section \ref{sec:PM}). 

The key part of the article is the analysis of \Krr's forcing in Section \ref{sec:krajicek}. In Theorem \ref{th:BBtype} we prove it is isomorphic to the probability algebra $\BK$ (under the assumption $|M| = \omega_1$). However, the forcing itself does not completely determine the resulting extensions: the essential part of the arguments in \cite{krajicek11} is the appropriate selection of the set of random variables $\mathcal R$, which form the domain of the generic extensions, based on some restrictions regarding their propositional complexity properties. The resulting generic models $\MGF$ are therefore parametrized by $\mathcal R$, the set of random variables (if $\mathcal R$ contains the identity function $\id: \Omega \to \Omega$, then the new number determined by $\id$, the ``random integer'',  constructs the whole model $\MGF$ in the sense of Theorem \ref{th:st1}).  

If $\mathcal R$ is maximal (it contains all random variables), the corresponding model $\MGa$ is the ``maximal'' model, with the other models $\MGF$ being its substructures (Theorem \ref{th:reg}). We do not have space to study the structure of the models $\MGF$ in detail, but we think it is plausible that it corresponds to the computational complexity properties of the respective Boolean-valued models in \cite{krajicek11} in a non-trivial way. We limit ourselves to proving several results related to the density of $M$ in the linear order $(\MGF,<)$, and conversely, see Theorems \ref{th:dd} and \ref{th:dense}. 

We end the article by discussing the topic of forcing in bounded arithmetics and its relationship to forcing in set theory from a broader perspective (Section \ref{sec:found}).

\brm
We do not apply our results to obtain new independence theorems for bounded arithmetics, our focus is on the set-theoretic properties of the generic extensions $\MGF$. However, with more work, it is plausible that connections can be found (because properties of $\MGF$ capture forcible statements in the sense of Lemma \ref{fact:m} and the recursive definition before (\ref{k:1})). See also Remark \ref{rm:bound} for more comments.
\erm

\brm A remark on notation. We use $\N$ to denote the standard natural numbers and we use the index $k$ to range over $\N$. We use letters $m,n, \Omega$ to denote non-standard numbers.  Note that $\N$ is usually denoted by $\omega$ in set theory, but we prefer to use $\N$ since we are primarily interested in models of arithmetics where $\N$ is more common.
\erm

\textbf{Acknowledgement.} The author wishes to thank (alphabetically) to Ond{\v r}ej Je{\v z}il, Jan \Kr, Moritz M{\"u}ller and Mykyta Narusevych for inspiring discussions on the relationship between forcing in set theory and arithmetics. Any inaccuracies in the article are solely my own, though.

\section{Models of arithmetics in models of set theory}\label{sec:1}

We start by observing that one can consider non-standard models of arithmetics inside of a transitive model of set theory:  Let $(V,\in)$ be a transitive countable model of set theory.\footnote{We use the letter $V$ to denote a transitive countable model of set theory. Alternatively, we can forget about $V$ and view the objects $(M,\mathcal R)$ and $\P$ as being defined in an ambient transitive model of set theory. Only at the stage when we consider a generic filter, we may say that the ambient universe is actually countable in some still larger universe. See Kunen \cite[Section IV.5]{MR2905394} for an extensive discussion of metamathematical approaches to forcing.} Suppose 
$$(M, \mathcal R) \in V$$ is a structure containing a non-standard model of arithmetics $M$ (as a structure in a suitable language extending the usual language of Peano Arithmetics) and some additional parameters $\mathcal R$. Let $(\P,\preceq) \in V$ be a partial order whose size may be uncountable in the sense of $V$. Using the standard set-theoretic mechanism, let $V^\P$ be the set of names in $V$. If $G$ is a generic filter which meets all dense sets in $V$, let $V[G]$ denote the respective forcing extension.

A ``$\P$-generic extension'' of $M$ can be defined by means of an appropriately chosen name $\dotM \in V^\P$, which depends on $M$ and the parameters in $\mathcal R$. Provided that the generic structure is uniquely defined from $(M,\mathcal R)$,  the name $\dotM$ can always be chosen canonically in the sense that the weakest condition in $\P$ forces that $\dotM$ is a structure of the desired type.

Suppose $G$ is $V$-generic for $\P$, then we can define:

\begin{equation}\label{def:MG}
\dotMG :=_{df} \set{\sigma_G}{1_\P \Vdash \sigma \in \dotM},
\end{equation}

and call $\dotMG$ a generic extension of $M$ (determined by $M,\mathcal R$ and $G$).

\brm \label{rm: canonical}
We will discuss two examples of this general method. In Section \ref{sec:PW}, the Paris--Wilkie extension corresponds to a model $\dotMG$ of the form $(M,G)$, where $G$ is treated as an additional unary predicate. In Section \ref{sec:krajicek}, in particular in (\ref{m:1}), $\dotMG$ is determined by the model $M$ and a fixed collection $\mathcal R$ of random variables which determine the domain of the generic extension.
\erm

This set-up allows us to formulate a forcing completeness theorem for $\dotM$ and the associated generic extension as a straightforward corollary of the forcing completeness theorem for $V$ and $V[G]$ in set theory:

\begin{theorem}[A forcing completeness theorem] \label{th:complete}
The following are equivalent for all $\varphi$, $p \in \P$ and $\sigma$ such that $1_\P \Vdash \sigma \in \dotM$:
\bce[(i)]
\item $p \Vdash \varphi^{\dotM}(\sigma)$.
\item For all $V$-generic filters $G$ for $\P$ which contain $p$, $$V[G] \models \varphi^{\dotMG}(\sigma_G).$$
\ece
\end{theorem}

\begin{proof}
Here, $\varphi^{\dotM}$ is shorthand for the quantification in $\varphi$ bounded into the name $\dotM$. By the forcing completeness theorem, and by using $(\dotM)_G = \dotMG$, we obtain the desired equivalence.
\end{proof}

\brm
One may ask in what sense $M$ can be taken to be the ``ground model'' for the extension $\dotMG$. If $\dotM$ contains names for all elements of $M$, then one may think about $M$ as being the ground model because $M \sub \dotMG$. However, in contrast to set theory and the inclusion $V^\P \sub V$, in the present case the inclusion typically fails, obtaining $\dotM \not \sub M$. The reason is that the cumulative hierarchy of names does not have a natural interpretation in $M$.\footnote{In the case of random forcing in Section \ref{sec:krajicek}, while the random variables are included in $M$, $\mathcal R \sub M$, the names corresponding to $\mathcal R$ use the probability algebra $\BB$ which is disjoint from $M$. So, strictly speaking, $M \cap \dotM = \emptyset$, provided the usual definitions of names are used. However, it is possible to code names in some other way by elements in $M$ (one can identify names directly with the variables in $\mathcal R$). See Atserias and M{\"u}ller \cite{atseriasmuller15} for an axiomatic formulation of a forcing framework where this is done.}  In particular, while $V[G]$ is the $\sub$-least transitive model of $\ZFC$ with the same height as $V$ which contains $G$, $\MGF$ is the $\sub$-least model of a given theory which is constructible in $V[G]$ from $\dotM$ and $G$.
\erm

\brm
Suppose $\MGF$ is a model of a theory $T$ and satisfies some formula $\varphi$, i.e., $$V[G] \models \varphi^{\MGF}.$$ Then $T$ does not prove $\neg \varphi$: any proof of $\neg \varphi$ from $T$ in $V$ would be in $V[G]$, which would contradict the fact that in $V[G]$, $\MGF$ models $T \cup \{\varphi\}$.\footnote{Because $V$ and $V[G]$ are transitive models, the natural numbers are absolute, and so are the proofs.}
\erm

\brm \label{rm:bound}
Note that by Theorem \ref{th:complete}, the statements which one needs to force to obtain independence results in bounded arithmetics are formulas of the form $\varphi^{\dotM}$. If $\dotM$ denotes the standard model, this bounded formula is $\Delta_1^{\ZF}$ and hence absolute for transitive models (see Footnote \ref{ft:delta} for more details).  If $M$ is non-standard,  $\varphi^{\dotM}$ is no longer absolute. Still, the quantifiers are bounded by $\dotM$, and hence only the ``first-order'' properties of the extensions $\MGF$ are relevant\footnote{Boolean-valued structures constructed in \Kr \cite{krajicek11} are often two-sorted, but the second sort for random functionals is typically expressible in terms of $M$, so from the point of set theory the relevant properties of the two-sorted structures $(\MGF, \ldots)$ are still only first-order. See Remark \ref{rm:2nd} for more details.}, making any second-order properties of models $\MGF$ interesting only from the point of set theory or model theory. But the dividing line is not clear: in  Remark \ref{rm:log} we mention a natural question in bounded arithmetics related to new numbers in $\MGF$ (the ``lengths'' of formulas), for which a better understanding of the extension $(\MGF,<)$ may prove to be fruitful (see Theorems \ref{th:st1}, \ref{th:dd} and \ref{th:dense}).
\erm

We will consider two examples to illustrate this approach: the Paris--Wilkie forcing with finite conditions in Section \ref{sec:PW} and the forcing with measure algebras in Section \ref{sec:krajicek}.

\section{Cohen reals: Paris--Wilkie forcing}\label{sec:PW}\label{sec:2}

Let $M \in V$ be a countable non-standard model of arithmetics and $n \in M$ is a non-standard number. The forcing notion $(\P,\preceq)$ is defined to contain as conditions all finite  functions from $n$ to $n+1$.\footnote{The generic filter adds a surjection from $n$ onto $n+1$, and hence an injection from $n+1$ into $n$.} The ordering is the reverse inclusion. Let $G$ be a $\P$-generic filter over $V$.

\begin{lemma}
$(\P,\preceq)$ is forcing-equivalent to a Cohen forcing adding a new subset of natural numbers $\N$.
\end{lemma}

\begin{proof}
This is a standard argument using the fact that there is up to isomorphism only one atomless countable Boolean algebra. If $A,B$ are countable, then a partial order composed of finite functions from $A$ to $B$ is densely embeddable into this Boolean algebra.\end{proof}

By results of Paris and Wilkie (see \cite{pariswilkiewoods88}, and also \cite[Section 4.1]{atseriasmuller15}), structure $(M,G)$, where $G$ is treated as an additional predicate, is a model of the negation of the pigeon hole principle for $n+1$ and $n$ together with the least number principle for existentially quantified quantifier-free formulas. 

\brm \label{rm:Cohen}
It is worth mentioning, for comparison with random forcing, that the Cohen forcing can be equivalently described as a forcing notion $\P$ whose conditions are closed intervals $[a,b]$, $a \neq b$, with rational endpoints in the unit interval $[0,1]$ on the reals. The ordering is $[a,b] \le [c,d]$ if and only if $[a,b] \sub [c,d]$. If $G$ is a generic filter, the intersection of the intervals in $G$ determines a new ``Cohen real''. Compare with the random forcing in Remark \ref{rm:mex}(\ref{mex:1}) which uses Borel subsets of positive measure, with the ordering being inclusion modulo the ideal of the Lebesgue null sets, to determine a new ``random real''.
\erm

\section{Probability measure algebras and forcing}\label{sec:PM}

\subsection{Maharam's theorem}\label{sec:M}

We will review Maharam's result that up to isomorphism there are very few infinite \emph{homogeneous probability measure algebra}: in fact any such algebra is isomorphic as a measure algebra to the algebra obtained from the standard product measure on the space $2^\kappa$ for some infinite cardinal $\kappa$, modulo the ideal of the null sets.

We will review basic facts and definition and observe that the measure algebra in \Kr \cite{krajicek11} is of this type (Theorem \ref{th:BBtype}). It follows that---despite its nonstandard construction---it is as a forcing notion equivalent to the usual set-theoretic forcing which adds $\kappa$-many new random reals (for an appropriate $\kappa > \omega$).

We will focus only on probability algebras.

\begin{definition}\label{def:ma}
A \emph{probability measure algebra} is a pair $\BM$ such that $\BMD$ is a $\sigma$-complete Boolean algebra and $\bar{\mu}: \BMD \to [0,1]$ is a function which satisfies:
\bce[(i)]
\item $\bar{\mu}(0) = 0$, $\bar{\mu}(a) > 0$ for all $a \neq 0$.
\item Whenever $\seq{a_k}{k \in \N}$ is a disjoint sequence in $\BMD$, then $\bar{\mu}(\bigvee \set{a_k}{k\in \N}) = \sum_{k \in \N}\bar{\mu}(a_k)$.
\item (Probability algebra) $\bar{\mu}(1) = 1$.
\ece
\end{definition}

Note that all measure algebras are ccc and complete (by ccc and $\sigma$-completeness).

A key notion in Maharam's analysis is that of \emph{Maharam-type-homogeneity} of a measure algebra $\BM$ (see \cite[331F]{fremlin2004measure}).

\begin{definition}\label{def:homo}
Let $\BM$ be a probability measure algebra.
\bce[(i)]
\item We say that $\BM$ is has \emph{Maharam type} $\kappa$ if $\kappa$ is the smallest cardinal such that there is a subset $D$ of $\BM$ of size $\kappa$ which \emph{completely generates} $\BMD$ (i.e.\ $\BM$ is the smallest algebra containing $D$ which is closed under suprema and infima of arbitrary sets). We denote the Maharam type of $\BM$ by $\tau(\BM)$.

\item We say that $\BM$ is \emph{Maharam-type-homogeneous} if  every  non-trivial principal ideal of $\BMD$ has the same Maharam type as the whole algebra, i.e.\ for every $y \neq 0$ in $\BMD$,  $\tau(\BMD) = \tau(\mathcal{B}_y)$, where $\mathcal{B}_y := \set{x \in \mathcal{B}}{x \le y}$.
\ece
\end{definition}

In the context of $\sigma$-finite measure algebras (so in particular for the probability algebras), the following Fact provides another useful characterization of Maharam-type-homogeneity.

\begin{Fact}
\label{f:homo}
A probability algebra $\BM$ is Maharam-type-homogeneous if and only if $\BMD$ is isomorphic, as a Boolean algebra, to every non-trivial principal ideal of $\BMD$ (i.e.\ to ideals $\mathcal{B}_y = \set{x \in \mathcal{B}}{x \le y}$ for some $y \in \mathcal{B}$).
\end{Fact}

\begin{proof}
See \cite[331N]{fremlin2004measure} for a proof.
\end{proof}

Maharam theorem asserts that all infinite homogeneous probability algebras are isomorphic as measure algebras to the algebra obtained from the product measure space $(2^\kappa, \Sigma_\kappa, \lambda_\kappa)$, for some infinite $\kappa$ (see Jech \cite[Example 15.31]{JECHbook} or \cite[254J]{fremlin2004measure2} for the details on this measure).\footnote{Let $T$ be the set of all finite functions from $\kappa$ to $2$. Let $\Sigma_\kappa$ be a $\sigma$-algebra generated by the system $\set{S_t}{t \in T}$, where $S_t = \set{f \in 2^\kappa}{t \sub f}$. The product measure $\lambda_\kappa$ is the unique $\sigma$-additive measure such that $\lambda_\kappa(S_t) = \frac{1}{2^{|t|}}$.} 

Let us review how a probability measure space is used to define a probability measure algebra. For this definition, let us review that a (probability) measure space $(X,\Sigma,\mu)$ is defined analogously as a measure algebra in  Definition \ref{def:ma} except that $\mu$ is defined on $\Sigma$, which is a $\sigma$-algebra of sets (see Fremlin \cite[112A]{fremlin2004measure1} for more details).

\begin{definition}\label{def:alg}
Given a measure space $(X, \Sigma, \mu) $, the associated \emph{measure algebra} $(\mathcal{B}, \bar{\mu})$ is defined as the  the quotient Boolean algebra: $$\mathcal{B} = \Sigma / \mathcal{N},$$ where $\mathcal{N}$ is the $\sigma$-ideal of null sets ($\{A \in \Sigma \mid \mu(A) = 0\}$). The measure $\bar{\mu}$ is induced by $\mu$ and defined by $$\bar{\mu}([A]) = \mu(A).$$ We will often omit $\bar{\mu}$ and write just $\mathcal{B}$ if there is danger of confusion.
\end{definition} 

\begin{definition}
For an infinite cardinal $\kappa$, let $(\Bk,\bar{\lambda}_\kappa)$ be the probability measure algebra associated with the measure space $(2^\kappa,\Sigma_\kappa,\lambda_\kappa)$. We will often write just $\Bk$ to denote the probability algebra.
\end{definition}

Let us state a version of Maharam's theorem which we will use:

\begin{theorem}[Maharam \cite{maharam1942homogeneous}]\label{th:Mah}
Suppose $\BM$ is an infinite probability measure algebra which is homogeneous in the sense of Definition \ref{def:homo}. Then $\BM$ is isomorphic\footnote{An isomorphism $f$ between Boolean algebras which preserves the measure exactly, i.e.\ $\bar{\mu}(a) = \bar{\lambda}_\kappa(f(a))$.} as a measure algebra to the probability algebra $(\Bk, \bar{\lambda}_\kappa)$ for some infinite $\kappa$.
\end{theorem}

\begin{proof}
For the proof, see the main theorems in \cite[331I]{fremlin2004measure} and \cite[331L]{fremlin2004measure} .
\end{proof}

\brm \label{rm:finite}
Let us state a few remarks. First, let us say that $(X,\Sigma,\mu)$ is a \emph{finitely additive measure space} if it satisfies the standard definition but it is only required that $\Sigma$ is an algebra of sets (not necessarily a $\sigma$-algebra).

\bce[(1)]
\item If a probability measure algebra $\BM$ is finite, then it is isomorphic to the two-element algebra $\{0,1\}$.
\item $\mathcal{B}_\omega$ is isomorphic to the the probability algebra of Lebesgue-measurable subsets of the unit interval.
\item It is possible that a measure algebra $\BM$ is derived from just a finitely additive measure space $(X,\Sigma,\mu)$. An important example is the context of non-standard analysis and Loeb measures: the internally defined measure spaces are just finitely additive, but the associated quotient algebras are ($\sigma$-complete) measure algebras. The reason for this discrepancy that the requirement on $\Sigma$ being a $\sigma$-algebra of sets is rather restrictive: the suprema of countable subsets of $\Sigma$ in $(X,\Sigma,\mu)$ must correspond to the set-theoretic unions, which may not be elements of $\Sigma$. However, there may be some other elements in $\Sigma$ which become the suprema in the associated Boolean algebra $\BM$ and ensure that $\BM$ is $\sigma$-complete.\footnote{See Remark \ref{rm:extend} for an alternative approach.}  See the discussion in \Kr \cite[page 12]{krajicek11} and Sections \ref{sec:KLoeb} and \ref{sec:KMaharam} for more details and references.
\ece
\erm

In Section \ref{sec:KMaharam} we will observe that  the probability measure algebras used in \Kr \cite{krajicek11} are homogenous, and hence are equivalent as forcing notions to the algebra $\Bk$ for some uncountable $\kappa$.

\subsection{Forcing with probability measure algebras}\label{sec:determines}

\begin{definition}\label{def:plus}
Let us fix the following notation. Suppose $\BM$ is a homogeneous probability measure algebra obtained as in Definition \ref{def:alg} from a (finitely additive) measure space $(X,\Sigma,\mu)$. Then we use $$(\Sp,\sub^*)$$ to denote the preorder which is forcing-equivalent to forcing with $\BM$: $\Sp := X - \mathcal{N}_\mu$ and $x \sub^* y \mbox{ iff } x - y \in \mathcal{N}_\mu$, for all $x, y \in X$. We will often use $\Sp$ to denote the preorder $(\Sp,\sub^*)$ (omitting also an explicit mention of $\mu$ which defines $\sub^*$) if there is not danger of confusion.
\end{definition}

Let us discuss the nature of the generic object added by $\Sp$. Let $G$ be $V$-generic for $\Sp$.  The generic filter $G$ can be equivalently described as some (any) $\sub^*$-decreasing sequence of non-zero sets \begin{equation}\label{eq:seq} \vec{g} = \seq{g_k \in \Sp}{g_k \in G, 0 < \mu(g_k) \le 2^{-k}, k \in \N}\end{equation} in the sense that $V[G] = V[\vec{g}]$ (in $V[\vec{g}]$, $G$ is defined as the collection of all measurable sets which extend some $g_k$ in $\sub^*$). 
\brm \label{rm:mex} Let us discuss some natural examples relevant for us:

\bce[(1)]

\item \label{mex:1}For $\mathcal{B}_\omega$, the preorder $\Sp$ can be equivalently described as Borel subsets of the unit interval $[0,1]$ of non-zero Lebesgue measure. If $G$ is a generic filter and $\vec{g}$ is as in (\ref{eq:seq}), there is a unique real $r_G$ which is the limit of all (Cauchy) sequences $\seq{r_k}{k \in \N}$ such that $r_k \in g_\kappa$ for all $k \in \N$.\footnote{\label{C} Note that intervals of positive measure are not dense in $\Sp$ because there exist nowhere dense Borel sets $C \sub [0,1]$ of positive measure (for instance closed fat Cantor sets) such that for every interval $[r,s]$, with $r < s$, in the unit interval, the measure of $C \cap [r,s]$ is positive and strictly smaller than $|r-s|$ (hence $[r,s]$ is not $\sub^*$-below $C$). However $\vec{g}$ is without loss of generality composed of intervals because by a density argument there is for every $k$ a set $A \in G$ which is contained in an interval of measure $\le 2^{-k}$. Compare with the Cohen forcing in Remark \ref{rm:Cohen}.} The real $r_G$ is called the \emph{random real}; $r_G$ determines the whole extension $V[G]$ because $G$ is defined (in $V[r_G]$) as the set of Borel sets which contain a Cauchy sequence converging to $r_G$.

The terminology for ``randomness'' of $r_G$ is motivated by another equivalent presentation of this forcing: $\Sp$ is a $\sigma$-algebra of subsets of $2^\omega$ with respect to the product measure $\lambda_\omega$. The space $2^\omega$ can be viewed, in the standard probability interpretation, as the space of ``fair-coin'' tosses (of length $\omega$). In this interpretation $\seq{g_k}{k \in \N}$ determines a new, random, sequence of tosses.

\item For $\Bk$, $\kappa> \omega$, there two main interpretations which we will consider:

\bce[(a)]
\item $\Bk$ gives rise to a sequence of $\kappa$-many random reals (due to projections onto individual coordinates). This is the interpretation in Scott \cite{scott67} for the consistency of the failure of $\CH$, or in Moore \cite{MR1981870} where sequences of random reals have some other special properties (see also Footnote \ref{ft:all}).

\item As we will observe with more details in Sections \ref{sec:KLoeb} and \ref{sec:KMaharam}, if $\CH$ holds and $M$ is nonstandard $\omega_1$-saturated model of arithmetics of size $\omega_1$, then the construction in \Kr \cite{krajicek11} yields a finitely additive measure on some element $\Omega \in M$, with the resulting measure algebra being isomorphic to $\BK$. The generic filter can now be interpreted as determining a new \emph{random integer} added to $M$ (the number $\id_G$ from Theorem \ref{th:st1}, see also Remark \ref{rm:id}).
\ece

\ece
\erm

\section{Probability measure algebras in bounded arithmetics}\label{sec:krajicek}

\subsection{Loeb measures in non-standard analysis}\label{sec:KLoeb}

The idea of using non-standard models to define standard measure algebras was introduced by Loeb \cite{Loeb}.  So called \emph{Loeb measures} are defined in non-standard $\omega_1$-saturated universes $({}^*V,\in)$ via internal  ``counting measures'' on a non-standard natural number in ${}^*\N - \N$ in ${}^*V$. The details of non-standard analysis are beyond the scope of this article. Standard references for this topic are Loeb \cite{Loeb}, and surveys in Cutland \cite{Cutland} or the reference book Loeb--Wolf \cite{loeb2000nonstandard}. For a concise summary related to Maharam types, we recommend Jin and Keisler \cite{jin1992maharam} and Jin \cite{jin1999distinguishing}.

The context of \Kr \cite{krajicek11} is more restricted because the measures are not defined using the whole structure ${}^*V$, but only  the non-standard integer part ${}^*\N$. We will observe that the non-standard analysis results interesting for us still apply in this restricted context.

Suppose the Continuum Hypothesis ($\CH$) holds, and let $M$ be a non-standard $\omega_1$-saturated model of \emph{true arithmetics} of size $\omega_1$ and $\Omega \in M$  a nonstandard number.\footnote{The assumption of $\CH$ allows us to fix an $\omega_1$-saturated model of size $\omega_1$. Without $\CH$, the size of $M$ is at least $2^\omega$. To simplify the presentation, we will assume $\CH$ in what follows.}

\brm \label{rm:un}
By a general theorem, if $M,N$ are elementarily equivalent first-order structures of size $\kappa$ which are $\kappa$-saturated, then they are isomorphic (see \cite[Theorem 10.1.8]{MT:Hod}). It follows that $M$ is up to isomorphism unique model of true arithmetics which is $\omega_1$-saturated and of size $\omega_1$. We will not use this fact explicitly, but it is good to keep in mind that this fact gives the starting assumptions of the construction below a degree of canonicity.
\erm

\brm \label{rm:log} Since $M$ is a model of true arithmetics, there exists for every $\varphi$ in the language of $M$ an $M$-definable function $c_\varphi$ such that $c_\varphi(n)$ codes the set \begin{equation} \label{code} \set{m \in M}{M\models  m < n \mbox{ and } \varphi(m)},\end{equation} viewed as a sequence of zeros and ones (a characteristic function). If $A \sub M$ corresponds to  a set in (\ref{code}), it is customary to  write $A \in M$ instead of using the appropriate number $ c_\varphi(n) \in M$ and ``decode'' $A$ from it. We write $|A|$ to denote the ``length'' of $A$, i.e., the length of the sequence of zeros and ones defining $A$. The length $|A|$ is viewed as being logarithmically smaller than $A$ (or, more precisely, logarithmically smaller than $c_\varphi(n)$).  In some applications it is critical whether new ``lengths'' are added (see for instance Kraj{\'i}{\v c}ek \cite[p.5]{NPKrajicek}), i.e., whether an extension of $M$ is a \emph{weak end-extension}. We will observe in Lemma \ref{th:dd}  that the extensions $\MGF$ for typical $\mathcal R$'s are far from being weak end-extension because a new integer is inserted between every two ground-model integers with an infinite distance (externally).
\erm

Let \begin{equation}\label{zero} \Apzero= \set{A \sub \Omega}{A \in M} \end{equation} denote the set of all subsets of $\Omega$ which are finite from the point of $M$. Note that $\Apzero$ is not a $\sigma$-algebra of sets (otherwise the set of standard natural numbers $\N$ would be definable in $M$). 

\brm
In the non-standard analysis setting of ${}^*V$, $\Apzero$  is defined as a collection of ``hyperfinite subsets of $\Omega$'',  i.e., those which are finite from the point of ${}^*V$. The same concept is equivalently expressed in (\ref{zero}) by considering subsets of $\Omega$ coded by some element in $M$. Note that the coding can be iterated for every $k \in \N$, obtaining correspondence between $M$ and the union of internal powersets of $M$ iterated through $\N$ (we will not need this iteration here).
\erm

For $m < n \in M$, let us write $n - m$ for the $M$ cardinality of the interval with end-points $m$ and $n$ (equivalently,  $|[m,n]| = n-m$, where $[m,n]$ is viewed as a set).

\begin{definition}
Suppose $n\le m$ are non-standard numbers. Let us write $\st(\frac{n}{m})$ for the \emph{standard part} of the $M$-fraction $\frac{n}{m}$. It is defined as the unique real $r \in [0,1]$ which is the limit of a sequence $\seq{\frac{l_k}{k}}{ k, l_k \in \N, k \neq 0, l_k <k}$ such that for all $k \in \N$, $k \neq 0$, $\frac{l_k}{k} \le \frac{n}{m} \le \frac{l_k+1}{k}$.
\end{definition}

\begin{definition}\label{def:measure}
Let $\mup$ be the \emph{counting measure} which assigns to each $A \in \Apzero$ the standard part of of the $M$-rational $\frac{|A|}{|\Omega|}$, i.e.\ $$\mup(A) = \st(\frac{|A|}{|\Omega|}).$$
\end{definition}

\begin{definition}
Let $$\Iup = \set{A \in \Apzero}{\mup(A) = 0}.$$
\end{definition}

Notice that $\Iup$ is equivalently defined as the set of all $A \in \Apzero$ such that $$\frac{|A|}{|\Omega|} < \frac{1}{k} \mbox{, equivalently }k\cdot |A| < |\Omega|, \mbox{ for every $k \in \N, k \neq 0$.}$$ 

\begin{definition}
If $A \in \Iup$, we sometimes say that the number $|A|$ is \emph{infinitesimally small}, or just \emph{infinitesimal} (with respect to $\mup$).
\end{definition}

$\Iup$ is not a $\sigma$-ideal: all standard initial segments $[0,k]$ for $k \in \N$ are in $\Iup$, but their union $\N$ is not an element of $M$.

By saturation of $M$, see \cite[Section 1.2]{krajicek11}, the quotient Boolean algebra \begin{equation}\label{eq:BB} \BB := \Apzero/ \Iup \end{equation} is a ccc Boolean algebra which is a probability measure algebra with respect to the measure $\mup$, extended to $\BB$:   $$\mup([A]_{\Iup}) := \mup(A).$$
Let \begin{equation}\label{eq:AA} (\Ap,\sub^*) \end{equation} denote the preorder corresponding to $\BB$ in the sense of Definition \ref{def:plus}.

\brm \label{rm:extend} Alternatively, it is possible to first extend the measure space $(M, \Apzero, \mup)$ into a $\sigma$-additive measure space (externally with respect to $M$), whose associated measure algebra is the same measure algebra as $\BB$. This intermediate step is not necessary but allows one to remain in the context of measure spaces (avoiding the use of measure algebras). See the discussion in Jin and Keiser \cite[\S 1]{jin1992maharam} (and the references mentioned there).
\erm

\brm
The measure $\mup$ is generated by $M$-finite intervals: Every  $A \in M$ is a union of intervals, and since $A \in M$, the number of intervals is an element of $M$ (i.e.\ it is finite from the point of $M$). However, the external cardinality of the number intervals composing a given $A \in M$ may be uncountable and, as we will observe in Theorem \ref{th:BBtype}, no countable subfamily of intervals suffices to generate the measure. This contrasts with the Lebesgue measure on the reals which is generated by countably many intervals with rational endpoints. In this sense, the probability algebra $\BB$ is \emph{non-separable}, due to Theorem \ref{th:BBtype}.
\erm

Let us end the section by stating a few facts about the linear order $(M,<)$ and the relationship with the measure $\mup$.

\begin{definition}
For $m<n \in M$, let us write $m \sim n$ if and only if $|n-m|$ is infinitesimal. Finally, let $[x]_\sim$ be the equivalence class of $x$ with respect to relation $\sim$.
\end{definition}

For every non-standard $x \in M$, let $\Z_x$ be the \emph{$\Z$-chain} generated by $x$ (the set $\set{x + k}{k \in \Z}$). Clearly, $\Z_x \sub [x]_\sim$ for every $x$. In Lemma \ref{lm:M-order}(\ref{m3}) we show that $[x]_\sim$ strictly bigger than $\Z_x$.\footnote{By Remark \ref{rm:un}, the properties of $(M,<)$ are probably well-known (even if the results are scattered in literature); however, we specifically state some which are relevant for our argument for the ease of reading.}

\begin{lemma}\label{lm:M-order}
Suppose $M$ is an $\omega_1$-saturated model of true arithmetics of size $\omega_1$ and $\Omega$ is as above. The following hold:
\bce[(i)]
\item The order $(M,<)$ is above the standard numbers composed of densely linearly ordered copies of $\Z$ (we call them $\Z$-chains). Moreover, for all non-standard $m < n$, with $n-m$ infinite, the external cardinality of the interval $[m,n]$ is $\omega_1$.
\item The linear order $([x]_\sim,<)$ on $\Omega$, with the order induced by $(M,<)$, is isomorphic via the standard part function to the unit interval on the reals with the standard order.
\item \label{m3a} The ideal $\Iup$ is closed under standardly finite unions but not under $M$-finite unions: If $m = \lfloor \sqrt{|\Omega|} \rfloor$, then the $M$-union of $m$-many pairwise disjoint subsets of $\Omega$ each of size $m$ has measure 1 even though each set has measure $0$ in $\mup$. More generally, whenever $mn + l = |\Omega|$ for some non-standard  numbers $m,n$, with $l <n$, then $m,n,l$ are infinitesimally small and $M$ can be in $M$ partitioned into $m$-many disjoint intervals $\set{I_i}{i<m}$ of size $n$, with a remainder set of size $l$, such that $\mup(I_i) = 0$ for all $i < m$ and $\mup(\set{A_i}{i<m}) = 0$. 
\item \label{m3} If $mn+l = |\Omega|$, for some non-standard $m,n$ and $l<n$, then for every $x \in \Omega$, $x \sim x+n'$, for $n' = \max(m,n)$. In particular, the equivalence classes $[x]_\sim$ are strictly bigger than the $\Z$-chains.
\ece
\end{lemma}

\begin{proof}
The items (i) and (ii) are standard, see for instance \cite{Cutland} for proofs (the external cardinality $\omega_1$ of $[m,n]$ with infinite distance follows from $\omega_1$-saturation and $|M| = \omega_1$).

Let us prove (\ref{m3a}). It is clear that $\Iup$ is closed under addition. Let show that it is not closed under multiplication. If $mn \le |\Omega|$ and $m,n$ are non-standard, then  $mk < |\Omega|$ and $nk < |\Omega|$ for every $k\in \N$, and hence $m,n$ are infinitesimal. Now, it is a theorem of arithmetics that whenever $0 \not = n < o$, then $o = mn + l$ for some $m$ and $l < n$ (division of $o$ by $n$ with remainder $l$),  and $o$ can be divided into $m$-many consecutive intervals of size $n$ with remainder of size $l$. The same must hold in $M$ since it is a model of true arithmetics. It follows there exists a  partition $\set{I_i}{i<m}$ definable inside $M$ as required by the lemma. In particular, if $m = \lfloor \sqrt{|\Omega|} \rfloor$ and $|\Omega| = m^2 + l$, then $m$ is infinitesimal, yet $\Omega$ can be partitioned into $m$-many intervals with size $m$ with a remainder of size $l$. Since $l < m$, the remained must be infinitesimal, and hence the measure of the union of $m$-many intervals of size $m$ must have measure $1$.

Item (\ref{m3}) is a consequence of (\ref{m3a}). \end{proof}

We will need Lemma \ref{lm:M-order}(\ref{m3}) later on in Theorems \ref{th:dd} and \ref{th:dense} when we analyse the extent of the mutual density of $(M,<)$ and its extension $(\MGF,<)$.

\brm
We will treat $\Omega$ as an initial segment of $M$ in the following sections because the more general situation can be easily reduced to this case: Every $\Omega$ used in applications in \Kr \cite{krajicek11} is a subset of $\max \Omega$ of positive measure. Forcing with $\Ap$ amounts to forcing with $\mathcal A^+_{M,[0,\max \Omega]}$ below the condition $\Omega$.  By Theorem \ref{th:BBtype}, both forcings are isomorphic with $\BK$ (and therefore with each other). The random variables defined on $[0,\Omega]$ can be naturally restricted to the smaller domain $\Omega$ as well.
\erm

\subsection{The Maharam type of \Kr probability algebra}\label{sec:KMaharam}

In the context of non-standard analysis mentioned in the previous section,  Jin and Keiser prove in \cite[Lemma 2.2]{jin1992maharam} that the Loeb measure defined on a non-standard natural number $\Omega$ is homogeneous, and in Theorem 2.1 that  its Maharam type is equal to the external cardinality of the internally computed power $2^{|\Omega|}$, where $|\Omega|$ is the size of $\Omega$ in $M$.

Let us transfer this result into the context of the measure algebra $\BB$.

\begin{lemma}\label{lm:homo}
Suppose $\BB$ a probability measure algebra defined in (\ref{eq:BB}). Then $\BB$ is Maharam-type homogeneous.
\end{lemma}

\begin{proof}
It is easy to check that the argument in Jin and Keisler \cite[Lemma 2.2]{jin1992maharam} proves the homogeneity of $\BB$ in the sense of Definition \ref{def:homo}. See also the proof of this lemma in Jin  \cite[Lemma 3.1]{jin1999distinguishing} which explicitly mentions Boolean (probability) algebras and moves between the measure spaces and the associated probability algebras.
\end{proof}

\begin{theorem}\label{th:BBtype}
Suppose $\CH$ holds and $M$ is an $\omega_1$-saturated model of true arithmetics of size $\omega_1$. Then the Maharam-type of $\BB$ is $\omega_1$, i.e.\ $\tau(\BB) = \omega_1$. It follows that $\BB$ is isomorphic to $\BK$ as a measure algebra.
\end{theorem}

\begin{proof}
By $\omega_1$-saturation of $M$, the external cardinality of $\Omega$ is uncountable, and hence equal to $\omega_1$ since $M$ has cardinality $\omega_1$. The proof in \cite[Theorem 2.1]{jin1992maharam} uses a probabilistic argument for natural numbers due to Shelah \cite[Lemma 2.3]{jin1992maharam}. Since $M$ is a model of true arithmetics, this result holds in $M$. It follows that the argument in \cite[Theorem 2.1]{jin1992maharam} goes through for $\BB$. Since the external cardinality of $2^{|\Omega|}$ is greater or equal to the external cardinality of $\Omega$, it follows that the external cardinality of $2^{|\Omega|}$ is $\omega_1$. It follows that $\tau(\BB) = \omega_1$, and hence $\BB$ is isomorphic as a measure algebra to $\BK$ by Maharam Theorem \ref{th:Mah}.
\end{proof}

Note that if $\GCH$ holds and $M$ is a model of true arithmetics which is $\kappa$-saturated of size $\kappa$, for a regular uncountable $\kappa$, then we similarly obtain that $\tau(\BB) = \kappa$, with $\BB$ being defined with respect to this $M$.

Thus in generic extensions for the forcing $\BS$ used by Scott \cite{scott67}, the following hold:\footnote{See Section \ref{sec:a} for the definition of generic models of bounded arithmetics in $V[G]$.}

\begin{corollary}\label{cor:BBtype}
Suppose $2^\omega = \omega_1$ and $2^\omega = \omega_2$ hold and let $G$ be a $V$-generic for $\BS$, then:
\bce[(i)]
\item $V[G]$ satisfies $2^\omega = \omega_2$.
\item If $\BB$ is defined as in (\ref{eq:BB}), then $G \cap \BK$ is $V$-generic for $\BK \cong \BB$. 
\item If $\BB$ is defined with respect a model of true arithmetics $M$ of size $\omega_2$ which is $\omega_2$-saturated, then $G$ is $V$-generic for $\BB$.
\ece
In particular the model $V[G]$ contains  non-standard models of size $\omega_1$ and $\omega_2$ of weak arithmetics which witnesses independence of statements discussed in \Kr \cite{krajicek11}.
\end{corollary}

Importantly, while the forcings are just of the form $\Bk$, they are used very differently in Scott \cite{scott67} and \Kr \cite{krajicek11}. The key difference is structure of the \emph{random variables} which add the new elements of the intended type: a new real number in \cite{scott67}, and a new non-standard integer in \cite{krajicek11}. A random variable in Scott \cite{scott67} is defined with respect to the  measure algebra $\mathcal{B}_\omega$. The algebra $\BS$ is used only as a product space to add an $\omega_2$-sequence of random reals added by $\mathcal{B}_\omega$ on each coordinate. In contrast, a random variable in \Kr \cite{krajicek11} is defined for the whole measure algebra $\BK$ simultaneously: one can view it as specifying where the new integers fit among the $\omega_1$-many ``ground-model'' integers in $M$ (see Remark \ref{rm:id} Theorems \ref{th:dd} and \ref{th:dense}).

\subsection{Random variables}\label{sec:a} \label{sec:Omega}

In this section, we review the definition and basic properties of random variables $\mathcal R$ from \Kr \cite{krajicek11} which can be naturally interpreted as names for the domain of the generic extensions $\MGF$.

\brm\label{rm:2nd}
Theories of bounded arithmetics are typically two-sorted. Models defined in \cite{krajicek11} are therefore two-sorted as well: they are composed of numbers, corresponding to \emph{random variables}, and relations on the domain, corresponding to \emph{random functionals}.\footnote{Functionals correspond to the second level of the cumulative hierarchy of names built on top of the first level of random variables. The same approach was implemented by Scott in \cite{scott67} who built a second-order theory of the reals $\R$ in which the Continuum Hypothesis problem is expressible.} We will consider only the first sort of objects, i.e., non-standard numbers, because this is sufficient for our purposes (and the inclusion of random functionals would make the article excessively long).  However, for applications in bounded arithmetics random functionals are essential (see Section \ref{sec:refine} for a few more details on random functionals).
\erm

\begin{definition}\label{def:Krandom}
We say that a function $$\alpha: \Omega \to M$$ is a \emph{random variable} if $$\alpha \in M.$$
\end{definition}

The assumption $\alpha \in M$ in particular implies that 

 \begin{equation} \label{k:atom}\set{x \in \Omega}{M \models \varphi(\alpha_0(x), \ldots, \alpha_{n-1}(x))} \in \Ac \end{equation} for every quantier-free formula $\varphi$ in the language $\mathcal L_{all}$ of arithmetics which includes the constants for $M$ and symbols for all relations and functions on $M$.

In applications, proper subsets of all random variables are used (see Section \ref{sec:ex2}).

\begin{definition}\label{def:F}
Let $\mathcal L$ be some fixed language (such as $\mathcal L_{all}$ mentioned in the previous paragraph). 
\bce[(i)]
\item Let $\Fa$ denote the collection of all random variables according to Definition \ref{def:Krandom}.
\item In general, let $\mathcal R$ a variable for a collection of random variables which satisfies the condition that is closed under all $\mathcal L$-functions and contains all $\mathcal L$-constants.
\ece
\end{definition}

The parameter $\mathcal R$ is not relevant for the definition of the forcing $\Ap$, but determines the resulting generic extensions of $M$, which we will denote $\MGF$ (see (\ref{m:1}) below).

Suppose $V, M$ and $\Omega$ are as above and let $\BB$ from (\ref{eq:BB}) be the corresponding probability measure algebra. Let $(\Ap,\sub^*)$ denote the corresponding partial order (\ref{eq:AA}). Suppose $\mathcal R$ is fixed as well. Let $G$ be a $V$-generic filter for $(\Ap, \sub^*)$. 

Working in $V[G]$, define an equivalence on random variables in $\mathcal R$ by \begin{equation}\label{def:a} \alpha \equiv \beta \mbox{ if and only if }\set{x \in \Omega}{\alpha(x) = \beta(x)} \in G.\end{equation}

Let us denote the equivalence classes determined by $\equiv$ by $\alpha_G$. Since these classes are in $ V[G]$, they have some $\BB$-names $\dot{\alpha}$. To keep the notation as simple as possible, we will identify random variables $\alpha$ with names for their respective equivalence classes. Boolean values $||\varphi||$ of formulas with names $\alpha, \ldots$ are defined recursively (\Kr \cite[page 13]{krajicek11}); using our notation with names, the definition reads as follows:

\bce[(i)]
\item $||\alpha = \beta|| := \set{x \in \Omega}{\alpha(x) = \beta(x)}/\Iup$.
\item $||R(\alpha_0, \ldots, \alpha_{k-1})|| := \set{x \in \Omega}{M \models R(\alpha_0, \ldots, \alpha_{k-1})}/\Iup$, for any $\mathcal L$-relation $A$.
\item $||\ldots||$ commutes with $\wedge, \vee, \neg$.
\item $||\exists x \varphi(x)|| := \bigvee_{\alpha \in F} ||\varphi(\alpha)||$.
\item $||\forall x \varphi(x)|| := \bigwedge_{\alpha \in F} ||\varphi(\alpha)||$.
\ece

Using the standard context of set-theoretic forcing, one can define the usual forcing relation as follows, for $p \in \BB$:

\begin{equation}\label{k:1} p \Vdash \varphi(\alpha, \ldots) \mbox{ if and only if } p \le ||\varphi(\alpha, \ldots)||.\end{equation}

Note that for every \emph{quantifier-free} formula $\varphi$, \begin{equation} \label{k:2} p \Vdash \varphi(\alpha, \ldots) \mbox{ if and only } p \le \set{x \in \Omega}{\varphi(\alpha(x), \ldots)}/\Iup.\end{equation}

Let us define a model $\MGF$ in $V[G]$ by postulating 
\begin{equation} \label{m:1} \MGF := \set{\alpha_G}{\alpha \mbox{ is a random variable in $\mathcal R$}}. \end{equation}

\Krr's analysis carried out in terms of Boolean values shows---using the correspondence in (\ref{k:1})---that $\MGF$ is a model of an appropriate theory of bounded arithmetics. 

For every $n \in M$ let $\alpha_n: \Omega \to M$ be constantly $n$. We identify the set $$\set{(\alpha_n)_G}{n \in M, \alpha_n \mbox{ is constantly }n}$$ with $M$. If $\mathcal R$ contains every constant function $\alpha_n$, we obtain $M \sub \MGF$.

\begin{definition}
We write $n$ for $(\alpha_n)_G$ if there is no danger of confusion.
\end{definition}

The following is a consequence of (\ref{k:2}), formulated in terms of generic models $\MGF$:

\begin{lemma}\label{fact:m}
Suppose the set $\set{x \in \Omega}{\varphi(\alpha_0(x), \alpha_1(x) \ldots)}$ is in $\Ap$ for some random variables $\alpha_0, \alpha_1, \ldots$ in $\mathcal R$, where $\varphi$ is a quantifier-free formula. Then if $G$ is $V$-generic for $\Ap$ and contains the condition $\set{x \in \Omega}{\varphi(\alpha_0(x), \alpha_1(x) \ldots)}$, then $$\MGF \models \varphi((\alpha_0)_G, (\alpha_1)_G \ldots).$$
\end{lemma}

We use Lemma \ref{fact:m} to prove a few results on the properties of the linear order $(\MGF,<)$ in Theorem \ref{th:dense} (i.e., we apply it to formulas $\varphi$ just in the language $\{<,=\}$).

Let us mention an important forcing-related corollary of the homogeneity of $\BB$ in Lemma \ref{lm:homo}. For the formulation, let us denote by $\dot{M}^{\mathcal R}$ a name for $\MGF$.

\begin{lemma}\label{c:homo}
The following are equivalent for every $p \in \BB$: $$1_{\BB} \Vdash (\exists \dotM) \varphi^{\dotM} \mbox{ if and only if } p \Vdash (\exists \dotM) \varphi^{\dotM},$$ where $\varphi$ is a sentence in the forcing language, and $\varphi^{\dotM}$ is its relativization to $\dotM$.
\end{lemma}

\begin{proof}
This follows because Lemma \ref{lm:homo} and Fact \ref{f:homo} guarantee an isomorphism between $\BB$ and $(\BB)_p$.
\end{proof}

Lemma \ref{c:homo} implies that with regard to independence results for bounded arithmetics, the exact choice of generic filters is irrelevant.

\subsection{The random integer $\id_G$} \label{sec:id}

Let us prove that the ``random integer'' $\id_G$ generates the whole generic extension $V[G]$.

\begin{theorem}\label{th:st1}  Suppose $G$ is a $V$-generic filter for $\Ap$. Then the following hold:
\bce[(i)]
\item $V[\id_G] = V[G],$ where $\id$ is the  random variable which is the identity on $\Omega$. 
\item If $\id \in \mathcal R$, $\id_G \in \MGF$ and therefore $\MGF$ is not an element of $V$.
\ece
\end{theorem}

\begin{proof}
Regarding (i). Define for every $A \in \Ap$ a random variable $\alpha^A: \Omega \to \Omega$ such that $\alpha^A$ is the identity on $A$ and $0$ otherwise. Then $$G = \set{A \in \Ap}{\id_G = \alpha^A_G},$$ and hence $V[G] = V[\id_G]$.

Regarding (ii). If $\id \in \mathcal R$, then $\id_G \in \MGF$ and hence $\MGF$ is in $V[G] \setminus V$.
\end{proof} 

\brm \label{rm:id}
Note that by an analogous argument, the \emph{random real $r_G$} determined by a generic for $\B_\omega$ in Remark \ref{rm:mex}(\ref{mex:1}) is identical to $\id_G$, where $\id$ is now the identity on the unit interval. In the analogy with $\B_\omega$, $\id_G$ from Theorem \ref{th:st1} may be called a \emph{random integer}. However, there are also some differences: while all distinct reals have a non-zero distance, and hence $r_G$ is a unique real for $\B_\omega$, there are  intervals on $(M,<)$ with measure zero. The \emph{random integer} should be more correctly understood as an interval on $(\MGF,<)$ with a measure zero which contains $\id_G$. More precisely, since the random variable $\id$ is unbounded in the terminology of Theorem \ref{th:dense}(\ref{d:2}), for every infinitesimally small $m$ in $M$, the interval $(\id_G-m, \id_G + m)$ on $(\MGF,<)$ does not contain any elements from $M$ and is therefore generated by $\id_G$.
\erm

\begin{lemma}\label{lm:st2}
Suppose the assumptions of Theorem \ref{th:st1} hold. Then:
\bce[(i)]
\item $\id_G < \max \Omega$.
\item \label{cof}  Suppose $\mathcal R$ contains all constant functions for elements of $M$. Then $(M,<)$ is a cofinal suborder of $(\MGF,<)$.
\ece
\end{lemma}

\begin{proof}
With regard to (i). Since $\id$ is the identity on $\Omega$, $\id_G \le \max \Omega$ by Lemma \ref{fact:m} ($\set{x \in \Omega}{\id(x) \le \max \Omega}$ is equal to $\Omega$ is therefore in every $G$). By a density argument, it cannot be equal to any ground-model number $n$, in particular $\id_G$ is not equal to $\max \Omega$ (see also Remark \ref{rm:new}).

With regard to (ii). Since every random variable $\alpha:\Omega \to M$ must be an element of $M$, the range of $\alpha$ must be bounded by some $n \in M$. If $n \in M$ is a such a bound, i.e.\ $\alpha(x) < n$ for every $x \in \Omega$, then $\alpha_G < n$.
\end{proof}

\subsection{The linear orders $(M,<)$ and $(\MGF,<)$}\label{sec:LO}

Suppose $\mathcal R$ is a collection of random variables as in Definition \ref{def:F} which contains all constant functions so that $(\MGF,<)$ extends $(M,<)$. Let prove several results on the structure of the linear order $(\MGF,<)$.

Let us first observe that new integers in $\MGF$ are added densely often below $\Omega$ in the strongest sense, provided $\mathcal R$ contains the appropriate random variables from Theorem \ref{th:dd}:\footnote{Clearly, $\Z$-chains are always formed by elements of the same type: either all elements in a $\Z$-chain are in the ground model  $M$ or new (in $\MGF \setminus M$). Hence if $|m-n|$ is standardly finite, there can be no new number inserted between them.}

\begin{theorem}\label{th:dd}
Suppose $m < n \in \Omega$ are two numbers with $|m-n|$ being externally infinite. If $G$ is $V$-generic for $\Ap$, then there is a random variable $\alpha$ such that $$m < \alpha_G < n.$$
\end{theorem}

\begin{proof}
Suppose first that $\mup([m,n]) = 0$. Let $p:= |(m,n)|$ and consider the division of $|\Omega|$ by $p$, obtaining a non-standard $l$ and $r < l$ such that $|\Omega| = pl +  r$. The devision determines a partition of $\Omega$ into $l$-many intervals of size $p$, $\set{I_i}{i<l}$,  with a remainder of size $r$. By arguments  in see Lemma \ref{lm:M-order}(\ref{m3a}), each $I_i$ has measure $0$ and the set of all intervals has also measure $0$. Define $\alpha$ to be a random variable in $M$ such that $\alpha$ is constant on the intervals $I_i$, $i < l$, and the induced map $\alpha"I_i$ is an injection from $\set{I_i}{i<l}$ into the interval $(m,n)$ ($\alpha$ can be defined arbitrarily on the remainder $r$). By Lemma \ref{fact:m}, $m < \alpha_G < n$. By the construction there is no $n' \in M$ such that $\set{x \in \Omega}{\alpha(x) = n'}$ has a positive measure, i.e., $\alpha_G$ is a new number.

Suppose now that $\mup([m,n]) > 0$. Then choose some subset $I \sub (m,n)$, $|m-n|$ infinite, with $\mup(I) = 0$, and proceed exactly as in the previous case.
\end{proof}

The converse direction, i.e.\ the density of $(M,<)$ in $(\MGF,<)$ below $\Omega$, is more complex because it depends on the types of random variables.

Notice that if $\alpha$ is a random variable in $M$, then the range of $\alpha$ is either standardly finite or else of uncountable cardinality by $\omega_1$-saturation of $M$ (since we assume $|M| = \omega_1$, the uncountable cardinality is $\omega_1$). If the range is finite, then $\alpha_G$ is a ground model number in $M$ by the additivity of the measure $\mup$. If the range is uncountable, the $\sigma$-additivity of $\mup$ is not sufficient by itself to determine whether $\alpha_G$ is a ground model number, and a more careful analysis must be carried out.

Let us say that $\alpha: \Omega \to \Omega$ is \emph{bounded} if there is an interval $I \sub \Omega$ of measure $0$, $I \in M$, such that $$\rng(\alpha) \sub I.$$ We say that $\alpha$ is bounded by $I$. 

Let us say that $\alpha: \Omega \to \Omega$ is \emph{unbounded} if for every interval $I \sub \Omega$ of measure $0$, $I \in M$, $$\alpha^{-1}"I \mbox{ has measure zero}.$$ A canonical example of an unbounded variable is the identity function $\id$ (see Remark \ref{rm:id} for more details on $\id_G$).

\brm \label{rm:new}
Note that if $\alpha$ is unbounded, then it is forced by the weakest condition in $\BB$ that it is not a number in $M$ (for every $m \in M$, $\set{x \in \Omega}{\alpha(x) = m}$ has measure zero).
\erm

This distinction is not a dichotomy in general, but one can naturally work below a condition in $\Ap$ where exactly one of the two options holds.

In the following theorem, items (\ref{d:1}) and (\ref{d:2}) provide conditions for an interval in $\MGF$ being disjoint from $M$, and items (\ref{l4}) and (\ref{j3}) for density of $M$ in $\MGF$.

For the arguments in the next theorem notice that a number $m^* \in M$ belongs to an interval $(\alpha_G,\beta_G)$ if and only if $\set{x \in \Omega}{\alpha(x) < m^* < \beta(x)} \in G$, where $G$ is a generic filter for $\Ap$. 

\brm
The results in (i) and (ii) of Theorem \ref{th:dense} are non-trivial. In particular for (i), if $\alpha$ is bounded by $I$ and $I$ is included in an interval $[m,n]$ on $M$, then $m < \alpha_G < n$, and hence the interval $(\alpha_G-m',\alpha_G + m')$, where $m' = |m-n|$, contains the number $n \in M$ (this can only happen if there is another $\beta_G \not \in M$ such that $\alpha_G+\beta_G= n$).
\erm

\begin{theorem}\label{th:dense} Let $V[G]$ be a fixed generic extension for $\Ap$.
\bce[(i)]

\item \label{d:1} Suppose $\alpha$ is bounded by $I$ of measure zero and $\alpha$ is forced by the weakest condition to be a new number.  Then there is a nonstandard infinitesimal $m < |I|$  such that the interval $(\alpha_G-m, \alpha_G + m)$ does not contain any numbers from $M$.

\item \label{d:2} Suppose $\alpha$ is an unbounded random variable. Then for every infinitesimal $m$ the interval $(\alpha_G-m, \alpha_G + m)$ does not contain any numbers from $M$.

\item \label{l4} Suppose $\alpha,\beta: \Omega \to \Omega$ are random variables. Suppose $\alpha_G < \beta_G$ are non-infinitesimally apart in the sense $$\mup(\set{x \in \Omega}{ \mup([\alpha(x),\beta(x)]) > 0}) > 0.$$ Then the interval $(\alpha_G, \beta_G)$ in $\MGF$ contains an interval of positive measure $I \sub \Omega$, $I \in M$.

\item \label{j3} Suppose $\alpha,\beta: \Omega \to M$ are random variables. Define that they are \emph{countably-apart} if whenever $A = \set{x \in \Omega}{\alpha(x) < \beta(x)}$ has a positive measure, then there is a countable collection $X_A \sub M$ such that the union of $A_m = \set{x \in A}{\alpha(x) \le m \le \beta(x)}, m \in X_A$, covers $A$. Then if $\alpha_G < \beta_G$, there is $m \in M$, such that $\alpha_G < m < \beta_G$. 
\ece
\end{theorem}

\begin{proof}
Regarding (\ref{d:1}). Since $\alpha$ is forced to be a new number by the weakest condition, the set $\set{I_x}{x \in \rng(\alpha) \sub I}$, where $I_x =  \alpha^{-1}{}"\{x\}$, is a partition of $\Omega$ into sets of measure zero. Let $r = |\rng(\alpha)|$. Let $m'$ be the maximum of the set $\set{|I_x|}{x \in \rng(\alpha)}$ (it exists because $\alpha \in M$). Note that $m'$ must be nonstandard (and infinitesimal): if $m' = k$ for some $k \in \N$, then $rk$ would be greater or equal to $\Omega$ which is impossible. Let $m < r$ be the maximum number such that $m'm^2 \le \Omega$. Since $m'$ is nonstandard infinitesimal, $m$ is nonstandard and $m'm$ is infinitesimal (because $m'mk<\Omega$ for all $k \in \N$). It is easy to check that for every $m^* \in M$, the set $\set{x \in \Omega}{m^* \in (\alpha(x)-m, \alpha(x) + m)}$ has measure zero because the number of elements in this set is bounded by $2m'm$. It follows that there is no ground model number in the interval $(\alpha_G-m, \alpha_G + m)$.

Regarding (\ref{d:2}). By Remark \ref{rm:new}, $\alpha_G$ is a new number. Let $m$ be an infinitesimal number. If $m^*$ is any number in $M$, the set $A_{m^*} := \set{x \in \Omega}{m^* \in (\alpha(x)-m, \alpha(x) + m)}$ has measure zero: since the intervals $(\alpha(x)-m,\alpha(x)+m)$ for $x \in A_{m^*}$ share a common number $m^*$, there is an interval $I$ of measure zero (in fact, of measure not larger than $2m$) such that $$\bigcup \set{(\alpha(x)-m,\alpha(x)+m)}{x \in A_{m^*}} \sub I.$$ By unboundedness of $\alpha$, $\alpha^{-1}"I$ has measure zero, and so does $A_{m^*}$. It follows that $m^*$ does not belong to the interval $(\alpha_G-m,\alpha_G+m)$.

Regarding (\ref{l4}). By the $\sigma$-additivity of $\mup$, we can assume that there is a positive $k \in \N$ such that 
$$\mup(\set{x \in \Omega}{ \mup([\alpha(x),\beta(x)]) \ge 1/k}) > 0.$$

Let $\set{B_i}{i < k}$ be a partition of $\Omega$ into $k$-many disjoint intervals each with measure at most $1/k$ (all $B_i$ are in $M$) obtained by dividing $\Omega$ by $k$: $$\Omega = [0,b_0] \cup \cdots \cup [b_{k-2},b_{k-1}],$$ where $[0,b_0] = B_0$, etc. By the $\sigma$-additivity of $\mup$, we can assume the there is a fixed $B_j$ and a positive $i \in \N$ such that 
\begin{equation} \label{i0} \mup(\set{x \in \Omega}{ \mup([\alpha(x),\beta(x)] \cap B_j) \ge 1/ki}) > 0.\end{equation}
Since we deal with intervals, we know that for each $[\alpha(x),\beta(x)]$ in the family in (\ref{i0}), $[\alpha(x),\beta(x)] \cap B_j$ determines an interval in $B_j$ which either contains the minimum or the maximum of $B_j$ (or both). Without loss of generality, assume that the first option yields a set of a positive measure, i.e., 
\begin{multline} \label{i1} \mup(\set{x \in \Omega}{ \mup([\alpha(x),\beta(x)] \cap B_j) \ge 1/ki \\ \mbox{ and the interval } [\alpha(x),\beta(x)] \cap B_j \mbox{ contains the min of }B_j}) > 0.\end{multline}
Let $I$ be an interval of $B_j$ which contains the minimum of $B_j$ and has measure at least $1/ki$. By our argument, it follows that $I \sub [\alpha(x),\beta(x)]$ for all $x$ in the family (\ref{i1}).

In the forcing language, the condition \begin{multline} \label{i2} \set{x \in \Omega}{ \mup([\alpha(x),\beta(x)] \cap B_j) \ge 1/ki \\ \mbox{ and the interval } [\alpha(x),\beta(x)] \cap B_j \mbox{ contains the min of }B_j}\end{multline} forces that $I$ is included in the interval $[\alpha,\beta]$, where $[\alpha,\beta]$ is a name for the interval in $\MGF$ determined by $\alpha_G$ and $\beta_G$ (where $G$ is a generic filter containing the condition (\ref{i2})).

Finally, regarding (\ref{j3}). Suppose for a contradiction that there are random variables $\alpha,\beta$ such that $$A = \set{x \in \Omega}{\alpha(x) < \beta(x)} \in \Ap,$$ but there are no $B \sub^*A$ and $m \in M$ such that \begin{equation} \label{contra} B = \set{x \in \Omega}{\alpha(x) \le m \le \beta(x)} \in \Ap.\end{equation} By our assumption, there is a countable collection $X_A\sub M$ such that the sets $$A_m = \set{x \in \Omega}{\alpha(x) \le m \le \beta(x)},$$ $m \in X_A$, cover $A$. By the $\sigma$-additivity of the measure on the probability algebra not all of the $A_m$'s can have measure $0$, and any $A_m$ with the associated $m$ are counterexamples to  (\ref{contra}).
\end{proof}

Let us mention some natural questions related to the properties of $\MGF$. 

\begin{question}\label{q:2}
Besides the density of $(M,<)$ in $(\MGF,<)$, what other concepts there are for measuring how ``large'' is $[0,\Omega] \cap M \sub \Omega$ in $\MGF$?
\end{question}

\begin{question}
Suppose $\mu_{\MGF,\Omega}$ is the measure defined on $\MGF$ in $V[G]$ (using the construction for $\mup$). Under which assumptions is it well-defined? What is the relationship between $\mup$ and $\mu_{\MGF,\Omega}$? 
\end{question}

\subsection{The maximal model and its submodels}\label{sec:ex2}

Recall that $\Fa$ is the maximal set of random variables in Definition \ref{def:F}.

\begin{theorem}\label{th:reg}
Suppose $V,M, \Omega$ are given and let $\Fa$ and $\mathcal R$ be sets of random variables on $\Omega$.  Suppose $G$ is $V$-generic for $\AcOm$. Then the generic extension $\MGF$ is a substructure of $\MGa$ (they satisfy the same atomic formulas with random variables in $\mathcal R$).
\end{theorem}

\begin{proof}
Theorem follows easily by Lemma \ref{fact:m}:

\begin{multline*}
\MGF \models \varphi((\alpha_0)_G, \ldots, (\alpha_k)_G) \mbox{ iff }\\
\set{x \in \Omega}{M \models \varphi(\alpha_0(x), \ldots, \alpha_k(x))} \in G \\\mbox{ iff } \MGa \models \varphi((\alpha_0)_G, \ldots, (\alpha_k)_G) \end{multline*}for every atomic $\varphi$ and random variables $\alpha_0, \ldots, \alpha_k$ from $\mathcal R$.
It follows that existential closures of quantifier-free formulas true in $\MGF$ are true in $\MGa$ as well.\footnote{This is the same collection of formulas as mentioned in Section \ref{sec:PW} with regard to Paris--Wilkie forcing.}
\end{proof}

In real applications in \cite{krajicek11}, proper subsets $\mathcal R$ of $\Fa$ are always used. An illustrative example of such an $\mathcal R$---relatively accessible as regards the necessary background in bounded arithmetics---is in \Kr \cite[Section 24.3]{krajicek11}; the relevant model $\MGF$ corresponding to this construction witnesses that under certain assumption the theory $\mathrm{PV}$ does not prove $S^1_2$.

All other models $\MGF$ in \cite{krajicek11} (more precisely their restriction to the number sort) are substructures of the maximal model $\MGa$. It is the appropriate choice of $\mathcal R$ which is key to proving independence results in bounded arithmetics by means of models $\MGF$. 

\subsection{Some generalizations}\label{sec:refine}

Let us mention two important generalizations of the ideas presented here (following examples in \Kr \cite{krajicek11}). We have not developed the forcing interpretation for them, but it might be of some interest and relevance (both from the point of bounded arithmetics and set theory).

\begin{itemize}
\item In our review, we omitted the second-sort random variables (functionals) because they do not add genuinely new ideas for the construction (from the point of this article). However, by a discussion in \cite[Section 5.4]{krajicek11}, the second-sort variables in the two-sorted structure $K(F,G)$ can actually be added in a more general way and may not be elements of $M$ (but it is required that they satisfy the axioms of identity). An application of this generalization has not been found, yet, though.

\item Another direction of a possible generalization of the basic setup is to consider an infinite sample space $\Omega$ which is not an element of $M$ (see \cite[Section 1.5]{krajicek11}).

\end{itemize}

\section{Some historical and foundational comments}\label{sec:found}

In this section we provide some more context as regards history of forcing in bounded arithmetics and also some foundational comments on the relationship between forcing in set theory and arithmetics (and other weak theories).

``Bounded arithmetics'' refer to weak sub theories of Peano Arithmetic in which the schema of induction is appropriately weakened; it is known that there is a natural correspondence between non-provability of certain statement in these theories and the existence of lower bounds for proofs of propositional tautologies (see \cite{buss85,krajicek95,cook2010logical} for more details). For example lower bounds on the proof size in the constant depth Frege system are based on forcing constructions in bounded arithmetic by Ajtai~\cite{ajtai1994complexity}, and some state-of-the-art Boolean circuit size lower bounds (see  \cite{chen2024symmetric,li2024symmetric}) connect to provability result in bounded arithmetic~\cite{jevrabek2004dual}. See Honzik et al.\ \cite{Rjoint} where these topics are elaborated.

A major obstacle to overcome in developing forcing for (bounded) arithmetics is that all transitive models of set theory compute the standard natural numbers $\N$ the same way, and hence all formulas quantified into $\N$ are absolute.\footnote{\label{ft:delta}More generally, by Shoenefield's theorem, no $\Sigma^1_2$-statements in the second-order arithmetics can change their truth value between $V$ and $V[G]$ (see \cite[Theorem 25.20]{JECHbook} for more details and more general formulations).} This implies that non-standard models must be used to obtain indepence of arithmetical formulas, and the resulting computations are more complex and bring new challenges.\footnote{For set theory, the development of forcing has shown that for many formulas $\varphi$  independent from $\ZFC$, there are standard (transitive) models of set theory $M,N$ which witness their independence (i.e., forcing extensions). This is an interesting discovery from the foundational perspective: there is a priori no reason why standard models should be sufficient to witness independence from $\ZFC$.} It may be the case that a method different from forcing will be found  to show independence of formulas for arithmetics (other than the auto-referential G{\"o}del-like arguments), but for the moment no obvious candidates for alternative methods are known (see Shelah \cite[Problem 2.4]{MR1962296} and \cite{MR3468189}).

The forcing constructions in bounded arithmetics are sometimes developed with a notion of an appropriately defined notion of a generic filter $G,$ which yields a generic extension $M[G]$ of a non-standard model $M$, see for instance Takeuti and Yasumoto \cite{MR1441106, MR1649065}. However, in the majority of cases (see \cite{atseriasmuller15} for an extensive list of references), the notion of genericity is often left only implicit, and the construction is formulated in terms of forcing conditions (and sometimes of Boolean-valued models as in \cite{krajicek11}).\footnote{\label{ft:all} See Kunen \cite[Section IV.5]{MR2905394} for an extensive discussion of the metamathematical aspects of forcing, in particular as regards the role of generic filters. We would like to add that while generic filters and models  are dispensable from the metamathematical perspective, they provide a ``model-theoretic'' feel and allow one to study various---seemingly unrelated logical independence arguments---in a combinatorial setting of a single model. Thus in a generic extension $V[G]$, where $G$ is generic for $\BK$, one always finds for instance (i) \emph{Sierpinski sets of reals} (Moore \cite[Theorem 3.2]{MR1981870}) and (ii) \emph{non-standard models} $M[G]$ of bounded arithmetics of size $\omega_1$ witnessing forcing constructions in \Kr \cite{krajicek11} (see Corollary \ref{cor:BBtype}). For both examples (i) and (ii), a different combinatorial argument must be found, specific for the problem at hand, even though the forcing notion and the generic extension are the same.} 

The varied constructions, which share the use of partial orders which approximate the desired ``generic objects'', where axiomatized by Atserias and  M{\"u}ller \cite{atseriasmuller15}, giving these constructions a uniform presentation. The axiomatization in \cite{atseriasmuller15} specifies  the minimal requirements for forcing: in particular, they consider only the pair $(M,\P)$,\footnote{In contrast to set theory, a non-trivial partial order $\P$ is never an element of $M$, nor is it definable in $M$. Ensuring the definability of the forcing relation is therefore a new and difficult problem in forcing in bounded arithmetics. See the discussion in \cite[Section 4.3]{atseriasmuller15} for the role of definability for Ajtai forcing.} where $M$ is a countable model and $\P$ is a forcing notion. The countably many dense sets to be met are given, roughly speaking, by definability over the pair $(M,\P)$. The requirement on the countability is connected to the Rasiowa--Sikorsky lemma which \cite{atseriasmuller15} invokes to obtain generic filters. Incidentally, a specific forcing construction which does not seem to be easily interpreted by \cite{atseriasmuller15} is \Krr's forcing which uses an uncountable partial order and Boolean-valued evaluations.\footnote{\Kr \cite{krajicek11} states a footnote on page 3 asserting that one can use the L{\"o}wenheim--Skolem theorem to collapse the random forcing to a countable domain (see also \cite{Rjoint} where this point is elaborated). It is not clear whether this collapse would make \Krr's forcing interpretable in \cite{atseriasmuller15}.} 

In terms of general frameworks for forcing in bounded arithmetics, the present article may be interpreted as providing an alternative framework, conceptually different from \cite{atseriasmuller15}. It uses the set-theoretic forcing for the general techniques and facts (like the ``forcing completness theorem'', Theorem \ref{th:complete}, or the development of forcing names), while putting specific forcing arguments in bounded arithmetics into the context of non-standard models of arithmetics, embedded in a standard model of set theory (in the sense of Footnote \ref{ft:all}). This of course does not solve any of the combinatorial problems by itself; it just makes forcing in arithmetics an instance of forcing in set theory, with the extra requirement of developing new combinatorial techniques related to non-standard models of arithmetics.

\bibliographystyle{amsplain}
\bibliography{main}
 \end{document}